\documentclass[12pt, reqno]{amsart}
\usepackage{amsmath, amsthm, amscd, amsfonts, amssymb, graphicx, color}
\usepackage[mathscr]{eucal}
\usepackage[bookmarksnumbered, colorlinks, plainpages]{hyperref}
\hypersetup{colorlinks=true,linkcolor=red, anchorcolor=green, citecolor=cyan, urlcolor=red, filecolor=magenta, pdftoolbar=true}

\newtheorem{theorem}{Theorem}[section]
\newtheorem{lemma}[theorem]{Lemma}
\newtheorem{proposition}[theorem]{Proposition}
\newtheorem{corollary}[theorem]{Corollary}
\theoremstyle{definition}

\newtheorem{example}[theorem]{Example}

\theoremstyle{remark}
\newtheorem{remark}{Remark}
\numberwithin{equation}{section}

\DeclareMathAlphabet\mathoo{U}{eur}{b}{n}

\DeclareMathOperator*{\esssup}{ess\,sup}

\begin{document}
\setcounter{page}{1}

\title[Integral operators with almost periodic kernels]{Inverse-closedness of subalgebras \\of integral operators \\with almost periodic kernels}

\author{E.~Yu. Guseva}
 \address{Department of System Analysis and Control,
Voronezh State University\\ 1, Universitetskaya Square, Voronezh 394036, Russia}

\email{\textcolor[rgb]{0.00,0.00,0.84}{elena.guseva.01.06@gmail.com}}

\author{V.~G. Kurbatov}
 \address{Department of Mathematical Physics,
Voronezh State University\\ 1, Universitetskaya Square, Voronezh 394036, Russia}
\email{\textcolor[rgb]{0.00,0.00,0.84}{kv51@inbox.ru}}



\date{\today}

\begin{abstract}
The integral operator of the form
\begin{equation*}
\bigl(Nu\bigr)(x)=\sum_{k=1}^\infty e^{i\langle\omega_k,x\rangle}
\int_{\mathbb R^c}n_k(x-y)\,u(y)\,dy
\end{equation*}
acting in $L_p(\mathbb R^c)$, $1\le p\le\infty$, is considered. It is assumed that $\omega_k\in\mathbb R^c$, $n_k\in L_1(\mathbb R^c)$, and
\begin{equation*}
\sum_{k=1}^\infty\lVert n_k\rVert_{L_1}<\infty.
\end{equation*}
We prove that if the operator $\mathbf1+N$ is invertible, then
$(\mathbf1+N)^{-1}=\mathbf1+M$, where $M$ is an integral operator possessing the analogous
representation.
\end{abstract}

\maketitle

\section*{Introduction}\label{s:Introduction}
This paper is devoted to almost periodic linear operators, i.~e., operators with almost periodic `coefficients' (we interpret the kernel of an integral operator as a kind of `coefficients'). Different properties of such operators were investigated in~\cite{Amerio-Prouse71,Avron-Simon83,
Bellissard-Bessis-Moussa82,Bellissard-Lima-Testard85,Bruno-Pankov-Tverdokhleb01,
Coburn-Moyer-Singer73,Deift-Simon83,Favard28,Levenshtam03:eng,
Levitan-Zhikov82:eng,Mukhamadiev71:eng,Mukhamadiev72a:eng,
Pankov81:eng,Pankov90:eng,Pankov13,Rabinovich-Roch-Silberman07,Shubin78:eng,Shubin79:eng,Simon82,
Slyusarchuk81,Slyusarchuk08,Wahlberg12} and other works. Equations and operators with almost periodic `coefficients' often arise in applications for the following reason. It is known~\cite[p.~8]{Levitan-Zhikov82:eng} that if the flow generated by a differential equation is equicontinuous and the image of a solution is relatively compact, then the solution is an almost periodic function. Therefore the linearization of the equation along such a solution gives an equation with almost periodic `coefficients'.

It is almost evident that the inverse of an almost periodic operator is also almost periodic (Theorem~\ref{t:6.5.2}). In this paper, we mainly discuss the structure of inverses of integral operators with almost periodic kernels. Our main result (Theorem~\ref{t:main2}) states that if the Fourier series of the kernel of an integral operator converges absolutely, then the kernel of the inverse operator possesses the same property.
As an auxiliary result, we prove that the inverse of an integral operator with an almost periodic kernel is also an integral operator with an almost periodic kernel (Theorem~\ref{t:main}).

It is natural to formulate and discuss these problems using the language of full~\cite[ch.~1,~\S~1.4]{Bourbaki_Theories_Spectrales:eng} or inverse closed~\cite[p.~183]{Grochenig10} subalgebras. The history of full subalgebras have its origin in Wiener's theorem on absolutely convergent Fourier series~\cite{Wiener32}. For further results related to inverse closed classes, see~\cite{Balan-Krishtal10,Baskakov90:eng,Baskakov97b:eng,Baskakov-Krishtal14,
Beltita-Beltita15,Bickel-Lindner11:eng,Farrell-Strohmer10,Feichtinger83,
Fendler-Grochenig-Leinert08,Fendler-Leinert16,Fernandez-Torres-Karlovich17,Goldstein99,
Grochenig10,Grochenig-Klotz10,Grochenig-Leinert06,Grochenig-Rzeszotnik-Strohmer10,
Krishtal-Okoudjou08,Kurbatov-MZ88:eng,Kurbatov-Kuznetsova16,Mantoiu15,Rabinovich-Roch04,Sun05,Sun11} and references therein.

The paper is organized as follows. In Section~\ref{s:Notation}, we recall and specify the notation and terminology. In Section~\ref{s:AP operators} we recall the definitions of an almost periodic function and an almost periodic operator. In Sections~\ref{s:APW},~\ref{s:int op in L infty}, and~\ref{s:CN_1} we describe the auxiliary facts which we use in the proof: the inverse closedness of the general algebra of almost periodic operators having an absolutely convergent Fourier series (Theorem~\ref{t:B_{APW} is full}), some technical results concerning integral operators acting in $L_\infty$,
and the inverse closedness of the algebra of integral operators whose kernels vary continuously in $L_1$-norm (Theorem~\ref{t:fin}). In Sections~\ref{s:N_1,AP} we prove the inverse closedness of the algebra of integral operators with almost periodic kernels (Theorem~\ref{t:main}); the proof essentially uses Theorem~\ref{t:fin}. Finally, in Section~\ref{s:APWi}, we prove the inverse closedness of the algebra of integral operators with almost periodic kernels whose Fourier series converges absolutely (Theorem~\ref{t:main2}); the proof is based on Theorems~\ref{t:B_{APW} is full} and~\ref{t:main}.

\section{General notation and terminology}\label{s:Notation}
Let $X$ and $Y$ be complex Banach spaces. We denote by $\mathoo B(X,Y)$ the space of all bounded
linear operators acting from $X$ to $Y$. If $X=Y$ we use the brief notation $\mathoo B(X)$.
We denote by $\mathbf1\in\mathoo B(X)$ the identity operator.

As usual, $\mathbb Z$ is the set of all integers and $\mathbb N$ is the set of all
positive integers.
Let $c\in\mathbb N$. The linear space $\mathbb
R^c$ is considered with the Euclidian norm $|\cdot|$ and the associated inner product $\langle\cdot,\cdot\cdot\rangle$.

Let $\mathbb E$ be a complex Banach space with the norm $|\cdot|$; in Sections~\ref{s:AP operators} and~\ref{s:APW} we assume that $\mathbb E$ is arbitrary, but in subsequent sections we assume that $\mathbb E$ is finite dimensional and Hilbert. We
denote by $\mathscr L_p=\mathscr L_p(\mathbb R^c,\mathbb E)$, $1\le p<\infty$, the space
of all measurable functions $u:\,\mathbb R^c\to\mathbb E$ bounded in the semi-norm
\begin{equation*}
\Vert u\Vert=\Vert u\Vert_{L_p}=\Bigl(\int_{\mathbb R^c}|u(x)|^p\,dx\Bigr)^{1/p},
\end{equation*}
and we denote by $\mathscr L_\infty=\mathscr L_\infty(\mathbb R^c,\mathbb E)$ the space
of all measurable essentially bounded functions $u:\,\mathbb R^c\to\mathbb E$ with the
semi-norm
\begin{equation*}
\Vert u\Vert=\Vert u\Vert_{L_\infty}=\esssup|u(x)|.
\end{equation*}
Finally, we denote by $L_p=L_p(\mathbb R^c)=L_p(\mathbb R^c,\mathbb E)$, $1\le p\le\infty$, the Banach
space of all classes of functions $u\in\mathscr L_p$ with the identification almost
everywhere. For more details, see~\cite{Bourbaki_Integration:fr}.
Usually they do not distinguish the spaces $\mathscr L_p$ and $L_p$.

For any $h\in\mathbb R^c$, we define the \emph{shift operator}
\begin{equation}\label{e:S_h}
\bigl(S_hu\bigr)(x)=u(x-h),
\end{equation}
and for any $\omega\in\mathbb R^c$, we define the \emph{oscillation {\rm(}modulation{\rm)} operator}
\begin{equation*}
\bigl(\Psi_\omega u\bigr)(x)=e^{i\langle\omega,x\rangle}u(x).
\end{equation*}
Clearly, the operators $S_h$ and $\Psi_\omega$ act in $L_p(\mathbb R^c,\mathbb E)$ and their norms equal 1.

\begin{proposition}\label{p:SPsi=PsiS}
For all $h,\omega\in\mathbb R^c$,
\begin{equation*}
S_h\Psi_{\omega}=e^{-i\langle\omega,h\rangle}\Psi_{\omega}S_{h}.
\end{equation*}
\end{proposition}
\begin{proof}
The proof is by direct calculations.
\end{proof}

All algebras~\cite{Bourbaki_Theories_Spectrales:eng} are considered over the field of complex numbers. A complete normed algebra is called \emph{Banach}. We denote the unit of an algebra by the symbol~$\mathbf1$. If an algebra has a unit, it is called \emph{unital}.

A subset $\mathoo R$ of an algebra $\mathoo B$ is called a \emph{subalgebra} if $\mathoo R$ is stable under the algebraic operations (addition, scalar multiplication, and multiplication), i.~e. $A+B,\lambda A,AB\in\mathoo R$ for all $A,B\in\mathoo R$ and $\lambda\in\mathbb C$.
If the unit $\mathbf1$ of an algebra $\mathoo B$ belongs to its subalgebra $\mathoo R$, then $\mathoo R$ is called a \emph{unital subalgebra}. Any non-unital subalgebra $\mathoo R$ of an algebra $\mathoo B$ can be extended to a unital subalgebra, which we denote by~$\widetilde{\mathoo R}$.

A unital subalgebra $\mathoo R$ of a unital algebra $\mathoo B$ is called \emph{full}~\cite[ch.~1,~\S~1.4]{Bourbaki_Theories_Spectrales:eng} or \emph{inverse closed}~\cite[p.~183]{Grochenig10} if every $B\in\mathoo R$ that is invertible in
$\mathoo B$ is also invertible in $\mathoo R$. This definition is equivalent to the following one: for any $B\in\mathoo R$, the existence of $B^{-1}\in\mathoo B$ such that $BB^{-1}=B^{-1}B=\mathbf 1$ implies that $B^{-1}\in\mathoo R$.

\section{Almost periodic functions and operators}\label{s:AP operators}
A subset of a Banach space is called \emph{relatively compact} if its closure is compact with respect to the norm topology.

Let $X$ be a Banach space, and $T$ be a topological space. We denote by $C(T,X)$ the Banach space of all bounded continuous functions $u:\,T\to X$ with the norm
\begin{equation*}
\Vert u\Vert=\Vert u\Vert_C=\sup_{x\in\mathbb R^c}\Vert u(x)\Vert.
\end{equation*}
The main example is the space $C=C(\mathbb R^c,X)$.
Clearly, the operators $S_h$, $h\in\mathbb R^c$, defined by formula~\eqref{e:S_h} act in $C$ and $\lVert S_h\rVert=1$.
A function $x\in C(\mathbb R^c,X)$ is called \emph{almost periodic}~\cite{Amerio-Prouse71,Dixmier:fr,Levitan-Zhikov82:eng,Pankov90:eng} if the set $\{\,S_hx:\,h\in\mathbb R^c\,\}\subseteq C$ is relatively compact.
We denote by $C_{AP}=C_{AP}(\mathbb R^c,X)$ the subset of $C$ consisting of all almost periodic functions. Clearly, $C_{AP}$ is a closed subspace of $C$.

An operator $T\in\mathoo B\bigl(L_p(\mathbb R^c)\bigr)$, $1\le p\le\infty$, is called~\cite{Coburn-Moyer-Singer73,Slyusarchuk81} \emph{almost periodic} if the family
\begin{equation}\label{e:T{h}}
T\{h\}=S_hTS_{-h},\qquad h\in\mathbb R^c,
\end{equation}
where $S_h$ is defined by formula~\eqref{e:S_h}, continuously (in the norm) depends on $h$ and is relatively compact in $\mathoo B\bigl(L_p(\mathbb R^c)\bigr)$.
We denote the set of all almost periodic operators $T\in\mathoo B\bigl(L_p(\mathbb R^c)\bigr)$ by $\mathoo B_{AP}(L_p)=\mathoo B_{AP}\bigl(L_p(\mathbb R^c)\bigr)$.

\begin{theorem}[{\rm see, e.g.,~\cite[Theorem~6.5.2]{Kurbatov99}}]\label{t:6.5.2}
The set $\mathoo B_{AP}(L_p)$ is a closed full subalgebra of the algebra $\mathoo B(L_p)$.
\end{theorem}
\begin{proof}
The proof immediately follows from the definition of an almost periodic operator.
\end{proof}

\section{Almost periodic operators\\ with absolutely convergent Fourier series}\label{s:APW}
In this Section we assume that $\mathbb E$ is an arbitrary Banach space.

We call an operator $A\in\mathoo B(L_p)$ \emph{shift invariant} if
\begin{equation*}
AS_h=S_hA,\qquad h\in\mathbb R^c.
\end{equation*}
We denote by $\mathoo A(L_p)$ the set of all shift invariant operators $A\in\mathoo B(L_p)$.

\begin{example}\label{ex:mathoo A}
(a) The \emph{integral convolution operator}
\begin{equation*}
\bigl(Gu\bigr)(x)=\int_{\mathbb R^c}g(x-y)u(y)\,dy,
\end{equation*}
where $g\in L_1\bigl(\mathbb R^c,\mathoo B(\mathbb E)\bigr)$, belongs to $\mathoo A(L_p)$, $1\le p\le\infty$, see, e.g.,~\cite[Corollary 4.4.11]{Kurbatov99}.
(b) More generally, the convolution operator
\begin{equation*}
\bigl(Tu\bigr)(x)=\int_{\mathbb R^c}d\mu(x-y)\,u(y),
\end{equation*}
with a bounded operator-valued measure $\mu$ on $\mathbb R^c$ belongs to $\mathoo A(L_p)$, $1\le p\le\infty$, see~\cite[Theorem 4.4.4]{Kurbatov99} for more details.
(c) Let the space $\mathbb E$ be Hilbert, $\mathcal F:\,L_2(\mathbb R^c)\to L_2(\mathbb R^c)$ be the Fourier transform, and $\xi\in L_\infty\bigl(\mathbb R^c,\mathoo B(\mathbb E)\bigr)$. Then the operator $u\mapsto \mathcal F^{-1}(\xi\cdot\mathcal Fu)$ belongs to~$\mathoo A(L_2)$.
\end{example}

\begin{proposition}\label{p:ooA is full}
The set $\mathoo A(L_p)$ is a full closed subalgebra of\/ $\mathoo B(L_p)$, $1\le p\le\infty$.
\end{proposition}
\begin{proof}
The proof is evident.
\end{proof}

\begin{proposition}\label{p:Psi A Psi}
Let $A\in\mathoo A(L_p)$, $1\le p\le\infty$. Then $\Psi_{-\omega}A\Psi_\omega\in\mathoo A(L_p)$ for any $\omega\in\mathbb R^c$.
\end{proposition}
\begin{proof}
The proof follows from Proposition~\ref{p:SPsi=PsiS}.
\end{proof}

We denote by $\mathoo B_{APW}(L_p)$, $1\le p\le\infty$, the set of all operators of the form
\begin{equation}\label{e:N=sum Psi A'}
K=\sum_{\omega\in\mathbb R^c}\Psi_{\omega}A_{\omega},
\end{equation}
where $A_{\omega}\in\mathoo A(L_p)$ and at most a countable number of operators $A_{\omega}$ are nonzero, with
\begin{equation*}
\sum_{\omega\in\mathbb R^c}\lVert A_{\omega}\rVert<\infty.
\end{equation*}
We call an operator $K\in\mathoo B_{APW}(L_p)$ an \emph{almost periodic operator with absolutely convergent Fourier series}.

\begin{proposition}\label{p:B_APW is full}
The set $\mathoo B_{APW}(L_p)$ is a subalgebra of\/ $\mathoo B(L_p)$, $1\le p\le\infty$.
\end{proposition}
\begin{proof}
The proof follows from Propositions~\ref{p:Psi A Psi} and~\ref{p:ooA is full}.
\end{proof}

For $K\in\mathoo B_{APW}(L_p)$, in accordance with~\eqref{e:T{h}}, we set
\begin{equation}\label{e:S_hKS_-h}
K\{h\}=S_hKS_{-h},\qquad h\in\mathbb R^c.
\end{equation}

\begin{proposition}\label{p:N(H)=}
Let an operator $K\in\mathoo B_{APW}(L_p)$ has the form~\eqref{e:N=sum Psi A'}. Then
\begin{equation}\label{e:K{h}}
K\{h\}=\sum_{\omega\in\mathbb R^c} e^{-i\langle\omega,h\rangle}\Psi_{\omega}A_{\omega},\qquad h\in\mathbb R^c.
\end{equation}
\end{proposition}
\begin{proof}
The proof follows from Proposition~\ref{p:SPsi=PsiS}.
\end{proof}

We denote by $\mathbb U$ the multiplicative group $\{\,z\in\mathbb C:\,|z|=1\,\}$. A function $\varkappa:\,\mathbb R^c\to\mathbb U$ is called~\cite[ch.~II, \S~1.1]{Bourbaki_Theories_Spectrales:eng},~\cite[22.15]{Hewitt-Ross-1:eng} a \emph{character} of the group $\mathbb R^c$ if
\begin{align*}
\varkappa(x+y)&=\varkappa(x)\varkappa(y),&x,y&\in\mathbb R^c,\\
|\varkappa(x)|&=1,&x&\in\mathbb R^c.
\end{align*}
We denote by $\mathbb X_b=\mathbb X_b(\mathbb R^c)$ the set of \emph{all} characters and call elements of $\mathbb X_b$ (\emph{discontinuous{\rm)} characters} of $\mathbb R^c$.
Sometimes we will denote the action of a character $\varkappa\in\mathbb X_b$ on $x\in\mathbb R^c$ by the symbol $\langle x,\varkappa\rangle$.

We denote by $\mathbb X=\mathbb X(\mathbb R^c)$ the set of all \emph{continuous} (with respect to the usual topology on $\mathbb R^c$) \emph{characters}. We endow $\mathbb X_b=\mathbb X_b(\mathbb R^c)$ with the topology of pointwise convergence. And we endow $\mathbb X=\mathbb X(\mathbb R^c)$ with the topology of uniform convergence on compact sets. Clearly, the topology on $\mathbb X_b=\mathbb X_b(\mathbb R^c)$ can also be interpreted as a topology of uniform convergent on compact sets provided one considers the group $\mathbb R^c$ with the discrete topology.

We define the sum of elements of $\mathbb X_b$ in the pointwise sense:
\begin{equation*}
(\varkappa_1+\varkappa_2)(x)=\varkappa_1(x)\varkappa_2(x).
\end{equation*}

\begin{proposition}[{\rm\cite[23.2]{Hewitt-Ross-1:eng}}]\label{p:X(G) is a group}
The set $\mathbb X_b=\mathbb X_b(\mathbb R^c)$ of all characters of\/ $\mathbb R^c$ is an abelian group, and $\mathbb X=\mathbb X(\mathbb R^c)$ is its subgroup.
\end{proposition}

Clearly, the zero element of the groups $\mathbb X_b$ and $\mathbb X$ is the function $0(x)\equiv1$.

\begin{proposition}[{\rm\cite[23.27f]{Hewitt-Ross-1:eng}}]\label{p:X(Rc)}
The group $\mathbb X(\mathbb R^c)$ is topologically isomorphic to the group $\mathbb R^c$. Namely, $\mathbb X=\mathbb X(\mathbb R^c)$ consists of the functions $\chi=\chi_\omega:\,\mathbb R^c\to\mathbb U$ of the form
\begin{equation*}
\chi(x)=\chi_\omega(x)=e^{i\langle x,\omega\rangle},
\end{equation*}
where $\omega$ runs over $\mathbb R^c$, and the mapping $J:\,\omega\mapsto\chi_\omega$ is an isomorphism.
\end{proposition}

\begin{corollary}\label{c:N(chi)=N(h)}
The function $K\{\cdot\}:\,\mathbb R^c\to\mathoo B(L_p)$ defined by rule~\eqref{e:K{h}} is equivalent to the function $K\{\cdot\}:\,\mathbb X(\mathbb R^c)\to\mathoo B(L_p)$ defined by the formula
\begin{equation}\label{e:7''}
K\{\chi\}=\sum_{\omega\in\mathbb R^c}\langle-\omega,\chi\rangle\Psi_{\omega}A_{\omega},\qquad \chi\in\mathbb X(\mathbb R^c).
\end{equation}
More precisely, the `equivalence' means that the diagram
\begin{equation*}
 \begin{CD}
\mathbb R^c @>J>>\mathbb X(\mathbb R^c)\\
@VK\{\cdot\}VV @VVK\{\cdot\}V\\
\mathoo B(L_p)@=\mathoo B(L_p)
\end{CD}.
\end{equation*}
is commutative {\rm(}see Proposition~\ref{p:X(Rc)} for the definition of $J${\rm)}.
\end{corollary}
\begin{proof}
The proof follows from Proposition~\ref{p:X(Rc)}.
\end{proof}

We denote by $\mathbb R^c_d$ the group $\mathbb R^c$ considered with the discrete topology.

\begin{proposition}\label{p:R_b is compact}
The group $\mathbb X_b(\mathbb R^c)$ is compact.
\end{proposition}
\begin{proof}
Clearly, the set $\mathbb X_b=\mathbb X_b(\mathbb R^c)$ can be interpreted as the set of all \emph{continuous} characters of $\mathbb R^c_d$. Now the statement follows from~\cite[Theorem 23.17]{Hewitt-Ross-1:eng}.
\end{proof}

For a more detailed structure of $\mathbb X_b=\mathbb X_b(\mathbb R^c)$, see, e.g.,~\cite[Example 4.2.1e]{Kurbatov99}.

\begin{proposition}[{\rm\cite[26.15]{Hewitt-Ross-1:eng}}]\label{p:Kroneker}
The subgroup $\mathbb X(\mathbb R^c)\simeq\mathbb R^c$ is dense in the group $\mathbb X_b(\mathbb R^c)$ {\rm(}in the topology of\/ $\mathbb X_b(\mathbb R^c)${\rm)}.
\end{proposition}

In accordance with Proposition~\ref{p:X(Rc)}, we identify the initial group $\mathbb R^c$ with $\mathbb X(\mathbb R^c)$; and according to Proposition~\ref{p:Kroneker} we extend the group $\mathbb R^c\simeq\mathbb X(\mathbb R^c)$
to the group $\mathbb X_b(\mathbb R^c)$; in the latter case we denote the extension $\mathbb X_b(\mathbb R^c)$ by $\mathbb R^c_b$. We denote by $C(\mathbb R^c_b,X)$ the Banach space of all continuous functions $u:\,\mathbb R^c_b\to X$ with the norm $\lVert u\rVert=\max_{t\in\mathbb R^c_b}\lVert u(t)\rVert$.

\begin{theorem}[{\rm\cite[16.2.1]{Dixmier:fr},~\cite[p. 7]{Pankov90:eng}}]\label{t:5.1.9}
Let $X$ be a Banach space. For a function $u\in C(\mathbb R^c,X)$, the following assumptions are equivalent.
\begin{enumerate}
 \item[{\rm(a)}] $u\in C_{AP}(\mathbb R^c,X)$.
 \item[{\rm(b)}] The function $u$ can be approximated in the norm of\/ $C(\mathbb R^c,X)$ by functions of the form
\begin{equation*}
p(t)=\sum_{k=1}^m\langle x,\chi_k\rangle u_k,
\end{equation*}
where $\chi_k\in\mathbb X(\mathbb R^c)$ and $u_k\in X$.
 \item[{\rm(c)}] The function $u$ possesses an extension to a function $u\in C(\mathbb R^c_b,X)$. By Proposition~\ref{p:Kroneker}, this extension is unique.
\end{enumerate}
\end{theorem}
In order not to complicate our notation, we use the same symbol for an initial function defined on $\mathbb R^c$ and for its extension to $\mathbb R^c_b$. We will always make it clear which function we mean by mentioning its domain.

\begin{corollary}\label{c:5.1.9}
Let $T\in\mathoo B_{AP}\bigl(L_p(\mathbb R^c)\bigr)$. Then family~\eqref{e:T{h}} possesses a unique extension to a function $T\{\cdot\}\in C\bigl(\mathbb R_b^c,\mathoo B\bigl(L_p(\mathbb R^c)\bigr)\bigr)$.
\end{corollary}
\begin{proof}
The proof follows from Theorem~\ref{t:5.1.9}.
\end{proof}

A \emph{positive measure} on $\mathbb X_b=\mathbb R^c_b$ is a bounded linear functional $\mu$ on $C(\mathbb X_b,\mathbb C)$ that is non-negative on non-negative functions $u\in C(\mathbb X_b,\mathbb C)$. A \emph{Haar measure} on the group $\mathbb X_b=\mathbb R^c_b$ is a non-zero shift invariant positive measure $\mu$, i.~e., $\mu(u)=\mu(S_\varkappa u)$ for all $u\in C(\mathbb X_b,\mathbb C)$ and $\varkappa\in\mathbb X_b$; here $\bigl(S_\varkappa u\bigr)(x)=u(x-\varkappa)$, $x\in\mathbb X_b$.

\begin{proposition}[{\rm\cite[Theorem 15.5]{Hewitt-Ross-1:eng}}]\label{p:Haar}
The Haar measure exists and is determined uniquely, up to a constant factor.
\end{proposition}

We normalize the Haar measure on $\mathbb X_b=\mathbb R^c_b$ so that $\int_{\mathbb X_b}\,d\varkappa=1$.

\begin{lemma}[{\rm see, e.g.,~\cite[Lemma 5.2.3]{Kurbatov99}}]\label{l:5.2.3}
For $h\in\mathbb R^c_d$ one has
\begin{equation*}
\int_{\mathbb X_b}\langle h,\varkappa\rangle\,d\varkappa=
\begin{cases}
 1&\text{if $h=0$},\\
0&\text{if $h\neq0$}.
\end{cases}
\end{equation*}
\end{lemma}

We consider the Banach algebra
$l_1=l_1\bigl(\mathbb R^c_d,\mathoo B(L_p)\bigr)$. The algebra $l_1$ consists of all families $\mathcal T=\{\,T_\omega\in\mathoo B(L_p):\,\omega\in\mathbb R^c\,\}$ such that the norm $\Vert\mathcal T \Vert=\sum_{\omega\in\mathbb R^c}\Vert T_\omega\Vert$ is finite (it is assumed that at most a countable number of operators $T_\omega$ are nonzero), with the point-wise linear operations. The multiplication in $l_1$ is defined as the convolution: the family $\mathcal R=\mathcal T*\mathcal S$, where $\mathcal T=\{\,T_\omega\in\mathoo B(L_p):\,\omega\in\mathbb R^c_d\,\}$ and $\mathcal S=\{\,S_\omega\in\mathoo B(L_p):\,\omega\in\mathbb R^c_d\,\}$, consists of the elements
\begin{equation*}
R_\omega=\sum_{\nu\in\mathbb R^c}T_{\omega-\nu}S_{\nu},\qquad\omega\in\mathbb R^c_d.
\end{equation*}
We will use the following variant of the Bochner--Phillips theorem~\cite{Bochner-Phillips}.

\begin{lemma}[{\rm\cite[Corollary 4.5.2(g)]{Kurbatov99}}]\label{l:Bochner--Phillips}
An element $\{\,T_h\in\mathoo B(L_p):\,h\in\mathbb R^c\,\}\in l_1$ is invertible in the algebra $l_1$ if and only if the operators $T\{\varkappa\}=\sum_{\omega\in\mathbb R^c}\langle\omega,\varkappa\rangle T_\omega$ are invertible for all $\varkappa\in\mathbb R^c_b$.
\end{lemma}

\begin{theorem}\label{t:B_{APW} is full}
The subalgebra $\mathoo B_{APW}(L_p)$, $1\le p\le\infty$, is full in the algebra $\mathoo B(L_p)$. More precisely, if an operator $K\in\mathoo B_{APW}(L_p)$ is invertible, then $K^{-1}=\sum_{\omega\in\mathbb R^c}\Psi_{-\omega} B_\omega$, where $B_\omega\in\mathcal A(L_p)$, with $\sum_{\omega\in\mathbb R^c}\lVert B_\omega\rVert<\infty$, and
\begin{equation*}
(K\{\varkappa\})^{-1}=\sum_{\omega\in\mathbb R^c}\langle-\omega,\varkappa\rangle \Psi_{-\omega} B_\omega,\qquad\varkappa\in\mathbb X_b(\mathbb R^c).
\end{equation*}
\end{theorem}

Theorem~\ref{t:B_{APW} is full} can be obtained as a special case of~\cite[Theorem 3.2]{Balan-Krishtal10}, see also~\cite{Baskakov90:eng} and~\cite[Theorem 5.2.5]{Kurbatov99}. Nevertheless, we give here a detailed proof for the completeness of the exposition.

\begin{proof}
Let an invertible operator $K\in\mathoo B_{APW}(L_p)$ be represented in the form~\eqref{e:N=sum Psi A'}.

For all $\varkappa\in\mathbb X_b(\mathbb R^c)$, we consider the operator
\begin{equation}\label{e:6'}
K\{\varkappa\}=\sum_{\omega\in\mathbb R^c}\langle-\omega,\varkappa\rangle\Psi_{\omega}A_{\omega}.
\end{equation}
It is clear that function~\eqref{e:6'} is an extension from $\mathbb X(\mathbb R^c)$ to $\mathbb X_b(\mathbb R^c)$ of the function $K\{\cdot\}$ defined by~\eqref{e:7''}. Moreover, since the series in~\eqref{e:6'} is absolutely convergent, function~\eqref{e:6'} is continuous with respect to the topology of $\mathbb X_b(\mathbb R^c)$. By Proposition~\ref{p:Kroneker}, it coincides with the extension of~\eqref{e:7''} by continuity.

Since the operator $K$ is invertible in the algebra $\mathoo B(L_p)$, from~\eqref{e:S_hKS_-h} and Corollary~\ref{c:N(chi)=N(h)} it is evident that the operators $K\{\chi\}$, $\chi\in\mathbb X(\mathbb R^c)$, are also invertible and $\Vert (K\{\chi\})^{-1}\Vert=\Vert K^{-1}\Vert$.
Because $\mathbb X(\mathbb R^c)$ is dense in $\mathbb X_b(\mathbb R^c)$ (see Proposition~\ref{p:Kroneker}), the operators $K\{\varkappa\}$ are invertible for all $\varkappa\in\mathbb X_b(\mathbb R^c)$ as well, with $\Vert (K\{\varkappa\})^{-1}\Vert=\Vert K^{-1}\Vert$.

We briefly denote the operators $\Psi_{\omega}A_{\omega}$ from~\eqref{e:N=sum Psi A'} by $T_\omega$ and consider the family
$\{\,T_{\omega}\in\mathoo B(L_p):\,h\in\mathbb R^c\,\}\in l_1\bigl(\mathbb R^c_d,\mathoo B(L_p)\bigr)$
associated with the operator $K$. As we have seen above, the operators
$K\{\varkappa\}$, $\varkappa\in\mathbb X_b(\mathbb R^c)$, are invertible. Hence, by Lemma~\ref{l:Bochner--Phillips}, the family $\{T_{\omega}\}$ is invertible in the algebra $l_1\bigl(\mathbb R^c_d,\mathoo B(L_p)\bigr)$;
we denote its inverse by $\{R_\omega\}\in l_1\bigl(\mathbb R^c_d,\mathoo B(L_p)\bigr)$. Thus,
\begin{equation*}
\sum_{\nu\in\mathbb R^c}T_{\omega-\nu}R_{\nu}=
\sum_{\nu\in\mathbb R^c}R_{\omega-\nu}T_{\nu}=
\begin{cases}
\mathbf1 & \text{if }\omega=0, \\
0 & \text{otherwise}.
\end{cases}
\end{equation*}

We show that the function $\varkappa\mapsto\sum_{\omega\in\mathbb R^c}\langle-\omega,\varkappa\rangle R_\omega$
is the point-wise inverse of the function $\varkappa\mapsto K\{\varkappa\}$.
In other words,
\begin{equation}\label{e:N{}-1}
(K\{\varkappa\})^{-1}=\sum_{\omega\in\mathbb R^c}\langle-\omega,\varkappa\rangle R_\omega
\end{equation}
for some $R_\omega\in\mathoo B(L_p)$, with $\sum_{\omega\in\mathbb R^c}\Vert R_\omega\Vert<\infty$. Indeed,
\begin{align*}
\Bigl(\sum_{\mu\in\mathbb R^c}\langle-\mu,\varkappa\rangle T_\mu\Bigr)
\Bigl(\sum_{\nu\in\mathbb R^c}\langle-\nu,\varkappa\rangle R_\nu\Bigr)
&=\sum_{\nu\in\mathbb R^c}\sum_{\mu\in\mathbb R^c}\langle-\mu-\nu,\varkappa\rangle T_\mu R_\nu\\
&=\sum_{\nu\in\mathbb R^c}\sum_{\omega\in\mathbb R^c}\langle-\omega,\varkappa\rangle T_{\omega-\nu} R_\nu\\
&=\sum_{\omega\in\mathbb R^c}\langle-\omega,\varkappa\rangle\sum_{\nu\in\mathbb R^c} T_{\omega-\nu} R_\nu\\
&=\sum_{\omega\in\mathbb R^c}\langle-\omega,\varkappa\rangle\begin{cases}
\mathbf1 & \text{if }\omega=0, \\
0 & \text{otherwise},
\end{cases}\\
&=\mathbf1.
\end{align*}
The equality $\Bigl(\sum_{\nu\in\mathbb R^c}\langle-\nu,\varkappa\rangle R_\nu\Bigr)\Bigl(\sum_{\mu\in\mathbb R^c}\langle-\mu,\varkappa\rangle T_\mu\Bigr)=\mathbf1$ is established in a similar way. In particular, substituting in these formulae $\varkappa=0$ we obtain
\begin{equation}\label{e:N-1=series}
K^{-1}=(K\{0\})^{-1}=\sum_{\omega\in\mathbb R^c}R_\omega.
\end{equation}

To complete the proof, we show that $R_\omega$ has the form $R_\omega=\Psi_{-\omega}B_{\omega}$, where $B_{\omega}\in\mathoo A(L_p)$.

It is straightforward to verify that
\begin{equation*}
K\{\varkappa+\chi_h\}=S_h K\{\varkappa\}S_h^{-1},
\qquad h\in\mathbb R^c,\;\varkappa\in\mathbb X_b,
\end{equation*}
From this identity, it follows that for an arbitrary $\omega_0$
\begin{equation*}
\langle\omega_0,\varkappa\rangle K\{\varkappa+\chi_h\}=
\langle\omega_0,\varkappa\rangle S_hK\{\varkappa\}S_h^{-1},
\qquad\varkappa\in\mathbb X_b.
\end{equation*}
We invert both the left and right sides:
\begin{equation*}
\langle-\omega_0,\varkappa\rangle(K\{\varkappa+\chi_h\})^{-1}=
\langle-\omega_0,\varkappa\rangle S_h(K\{\varkappa\})^{-1}S_h^{-1},
\qquad\varkappa\in\mathbb X_b.
\end{equation*}
Substituting representation~\eqref{e:N{}-1} is this formula, we obtain
\begin{equation*}
\sum_{\omega\in\mathbb R^c}\langle-\omega,\varkappa+\chi_h\rangle \langle-\omega_0,\varkappa\rangle R_\omega=
\langle-\omega_0,\varkappa\rangle S_h\sum_{\omega\in\mathbb R^c}\langle-\omega,\varkappa\rangle R_\omega S_h^{-1},
\qquad\omega\in\mathbb R^c,\varkappa\in\mathbb X_b.
\end{equation*}
Or
\begin{equation*}
\sum_{\omega\in\mathbb R^c}
\langle-\omega,\chi_h\rangle \langle-\omega-\omega_0,\varkappa\rangle R_\omega=
S_h\Bigl(\sum_{\omega\in\mathbb R^c}\langle-\omega-\omega_0,\varkappa\rangle R_\omega\Bigr) S_h^{-1},
\qquad\omega\in\mathbb R^c,\varkappa\in\mathbb X_b.
\end{equation*}
Next, we integrate the both sides with respect to $\varkappa\in\mathbb X_b$ (using the Haar measure and taking into account Lemma~\ref{l:5.2.3}):
\begin{multline*}
\int_{\mathbb X_b}\sum_{\omega\in\mathbb R^c}
\langle-\omega,\chi_h\rangle \langle-\omega-\omega_0,\varkappa\rangle R_\omega\,d\varkappa\\
=\int_{\mathbb X_b}S_h\Bigl(\sum_{\omega\in\mathbb R^c}\langle-\omega-\omega_0,\varkappa\rangle R_\omega\Bigr) S_h^{-1}\,d\varkappa,
\qquad\omega\in\mathbb R^c,
\end{multline*}
or
\begin{equation*}
\langle-\omega_0,\chi_h\rangle R_{-\omega_0}=S_hR_{-\omega_0}S_h^{-1},
\end{equation*}
or
\begin{equation*}
R_{-\omega_0}=e^{i\langle\omega_0,h \rangle}S_hR_{-\omega_0}S_h^{-1},
\end{equation*}
or
\begin{equation*}
R_{\omega_0}=e^{-i\langle\omega_0,h \rangle}S_hR_{\omega_0}S_h^{-1}
\end{equation*}
for all $\omega_0\in\mathbb R^c$.
By Proposition~\ref{p:SPsi=PsiS}, the last equality implies that
\begin{equation*}
\Psi_{\omega_0}R_{\omega_0}=S_h\Psi_{\omega_0}R_{\omega_0}S_h^{-1},
\end{equation*}
which means that $B_{\omega_0}=\Psi_{\omega_0}R_{\omega_0}\in\mathoo A(L_p)$.
\end{proof}

\begin{remark}\label{r:Mukhamadiev}
It is useful to notice that a one-sided invertibility of an almost periodic operator often implies its two-sided invertibility~\cite{Favard28,Kurbatov-85:eng,Kurbatov-MSb89:eng,
Mukhamadiev71:eng,Mukhamadiev72a:eng,Shubin78:eng,Slyusarchuk81,Slyusarchuk08}.
\end{remark}

\section{Integral operators in $L_\infty$}\label{s:int op in L infty}
In this and subsequent sections we assume that $\mathbb E$ is a finite-dimensional Hilbert space.

We denote by $\mathoo N_b=\mathoo N_b(\mathbb R^c,\mathbb E)$ the set of all measurable functions $n:\,\mathbb R^c\times\mathbb R^c\to\mathoo B(\mathbb E)$ such that for almost all $x\in\mathbb R^c$ the function $n(x,\cdot)$ (is defined almost everywhere and) belongs to $\mathcal L_1\bigl(\mathbb R^c,\mathoo B(\mathbb E)\bigr)$ and
\begin{equation*}
\lVert n\rVert_{\mathoo N_b}=\esssup_{x\in\mathbb R^c}\lVert n(x,\cdot)\rVert_{L_1}<\infty.
\end{equation*}
We introduce in $\mathoo N_b(\mathbb R^c,\mathbb E)$ the identification almost everywhere. After that $\mathoo N_b(\mathbb R^c,\mathbb E)$ becomes a normed space with the natural linear operations and the norm $\lVert\cdot\rVert_{\mathoo N_b}$.

 \begin{proposition}\label{p:N acts in L_p}
The function $n:\,\mathbb R^c\times\mathbb R^c\to\mathoo B(\mathbb E)$ is measurable if
and only if the function $n_1(x,y)=n(x,x-y)$ is measurable.
 \end{proposition}
 \begin{proof}
The proof is a word for word repetition of that of~\cite[Lemma 4.1.5]{Kurbatov99}.
 \end{proof}

For any $n\in\mathoo N_b(\mathbb R^c,\mathbb E)$, we denote by $\bar n$ the function that
assigns to each $x\in\mathbb R^c$ the function $\bar n(x):\,\mathbb R^c\to\mathoo
B(\mathbb E)$ defined by the rule
\begin{equation}\label{e:bar n}
\bar n(x)(y)=n(x,x-y);
\end{equation}
and for any $n\in\mathoo N_b(\mathbb R^c,\mathbb E)$, we denote by $\tilde n$ the function that assigns to each $x\in\mathbb R^c$ the function $\tilde n(x):\,\mathbb R^c\to\mathoo
B(\mathbb E)$ defined by the rule
\begin{equation}\label{e:tilde n}
\tilde n(x)(y)=n(x,y).
\end{equation}

\begin{proposition}\label{p:N_b kernel}
The normed space $\mathoo N_b(\mathbb R^c,\mathbb E)$ {\rm(}with the identification almost everywhere{\rm)} is isometrically isomorphic to $L_\infty\bigl(\mathbb R^c,L_1\bigl(\mathbb R^c,\mathoo B(\mathbb E)\bigr)\bigr)$.
\end{proposition}
\begin{proof}
Let $n\in\mathoo N_b(\mathbb R^c,\mathbb E)$. By the Fubini theorem~\cite{Bourbaki_Integration:fr}, the restriction of $n$ to the set $K\times\mathbb R^c$, where $K\subset\mathbb R^c$ is an arbitrary compact set, is summable. Therefore the restriction $\tilde n_K$ of $\tilde n$ to $K$ is a function of the class $L_1$. It is clear that actually $\tilde n\in L_\infty$ and $\lVert\tilde n\rVert_{L_\infty}=\lVert n\rVert_{\mathoo N_b}$.

Conversely, let $\tilde n\in L_\infty\bigl(\mathbb R^c,L_1\bigl(\mathbb R^c,\mathoo B(\mathbb E)\bigr)\bigr)$. Let $K\subset\mathbb R^c$ be compact. Clearly, the restriction $\tilde n_K$ of $\tilde n$ to $K$ belongs to $L_\infty\bigl(K,L_1\bigl(\mathbb R^c,\mathoo B(\mathbb E)\bigr)\bigr)$. Further, the space $L_\infty\bigl(K,L_1\bigl(\mathbb R^c,\mathoo B(\mathbb E)\bigr)\bigr)$ is included into $L_1\bigl(K,L_1\bigl(\mathbb R^c,\mathoo B(\mathbb E)\bigr)\bigr)$. It is known (see, e.g.,~\cite[Corollary 1.5.9]{Kurbatov99}) that the space $L_1\bigl(K,L_1\bigl(\mathbb R^c,\mathoo B(\mathbb E)\bigr)\bigr)$ is naturally isometrically isomorphic to $L_1\bigl(K\times\mathbb R^c,\mathoo B(\mathbb E)\bigr)$. Thus, we obtain a measurable function $n_K:\,K\times\mathbb R^c\to\mathoo B(\mathbb E)$. For almost all $x\in\mathbb R^c$ the function $n_K(x,\cdot)$ (is defined almost everywhere and) belongs to $\mathcal L_1\bigl(\mathbb R^c,\mathoo B(\mathbb E)\bigr)$ and
\begin{equation*}
\esssup_{x\in\mathbb R^c}\lVert n_K(x,\cdot)\rVert_{L_1}\le\lVert \tilde n\rVert_{L_\infty}.
\end{equation*}
Clearly, if $K\subseteq K_1$ are compact sets, the restriction of $n_{K_1}$ to $K\times\mathbb R^c$ coincides almost everywhere with $n_k$.
\end{proof}

\begin{proposition}\label{p:N_b operator}
For any $n\in\mathoo N_b(\mathbb R^c,\mathbb E)$, the operator
 \begin{equation}\label{e:operator N'}
\bigl(Nu\bigr)(x)=\int_{\mathbb R^c}n(x,x-y)\,u(y)\,dy
 \end{equation}
acts in $L_\infty(\mathbb R^c,\mathbb E)$. More precisely, for any
$u\in\mathscr L_\infty(\mathbb R^c,\mathbb E)$ the function $y\mapsto n(x,x-y)\,u(y)$ is
integrable for almost all $x$, and the function $Nu$ belongs to $\mathscr L_\infty(\mathbb
R^c,\mathbb E)${\rm;} if $u_1$ and $u_2$ coincide almost everywhere, then $Nu_1$ and
$Nu_2$ also coincide almost everywhere. Besides,
 \begin{equation}\label{e:norm of N_b}
c\lVert n\rVert_{\mathoo N_b}\le\Vert N:\,L_\infty\to L_\infty\Vert\le\esssup_{x\in\mathbb R^c}\lVert n(x,\cdot)\rVert_{L_1},
 \end{equation}
where $c$ is independent of $n$.
\end{proposition}
\begin{proof}
Let $K\subset\mathbb R^c$ be compact. Then the function $(x,y)\mapsto n(x,x-y)\,u(y)$, $x\in K$, $y\in\mathbb R^c$, belongs to $L_1\bigl(K\times\mathbb R^c,\mathoo B(\mathbb E)\bigr)$. Therefore, by the Fubini theorem, the function $y\mapsto n(x,x-y)\,u(y)$ is
integrable for almost all $x\in K$. The estimate
\begin{equation*}
\esssup_{x\in\mathbb R^c}\biggl\lVert \int_{\mathbb R^c}n(x,x-y)\,u(y)\,dy\biggr\rVert\le\esssup_{x\in\mathbb R^c}\lVert n(x,\cdot)\rVert_{L_1}
\end{equation*}
is evident.

The proof of the first part of estimate~\eqref{e:norm of N_b} repeats the proof of \cite[Proposition~3.7]{Kurbatov-Kuznetsova16}.
\end{proof}

 \begin{proposition}\label{p:ShNS-h:infty}
Let $n\in\mathoo N_b$, and the operator $N$ be defined by formula~\eqref{e:operator N'}.
Then for any $h\in\mathbb R^c$ the operator $S_hNS_{-h}$ is defined by the formula
\begin{equation*}
\bigl(S_hNS_{-h}u\bigr)(x)=\int_{\mathbb R^c}n(x-h,x-y)\,u(y)\,dy.
\end{equation*}
Thus, the operator $S_hNS_{-h}$ is generated by the kernel $(x,y)\mapsto n(x-h,y)$.
 \end{proposition}
  \begin{proof}
Indeed, we have
\begin{align*}
\bigl(NS_{-h}u\bigr)(x)&=\int_{\mathbb R^c}n(x,x-y)\,u(y+h)\,dy,\\
\bigl(S_hNS_{-h}u\bigr)(x)&=\int_{\mathbb R^c}n(x-h,x-h-y)\,u(y+h)\,dy\\
&=\int_{\mathbb R^c}n(x-h,x-y)\,u(y)\,dy.\qed
\end{align*}
\renewcommand\qed{}
 \end{proof}

We denote by $\mathoo N_b(L_\infty)$ the set of all operators induced by kernels $n\in\mathoo N_b(\mathbb R^c,\mathbb E)$.

\begin{corollary}\label{c:N_b iso}
The spaces $L_\infty\bigl(\mathbb R^c,L_1\bigl(\mathbb R^c,\mathoo B(\mathbb E)\bigr)\bigr)$ and $\mathoo N_b(L_\infty)$ are topologically isomorphic. Moreover, the natural correspondence $U:\,\tilde n\mapsto N$ maps $S_h\tilde n$ to $N\{h\}$ for all $h\in\mathbb R^c$. Thus, a function $\tilde n\in L_\infty\bigl(\mathbb R^c,L_1\bigl(\mathbb R^c,\mathoo B(\mathbb E)\bigr)\bigr)$ is almost periodic if and only if the associated operator $N\in\mathoo N_b(L_\infty)$ is almost periodic.
\end{corollary}
\begin{proof}
By Propositions~\ref{p:N_b kernel} and~\ref{p:N_b operator}, the correspondence $U:\,\tilde n\mapsto N$ is a topological isomorphism from $L_\infty\bigl(\mathbb R^c,L_1\bigl(\mathbb R^c,\mathoo B(\mathbb E)\bigr)\bigr)$ onto $\mathoo N_b(L_\infty)$. From Proposition~\ref{p:N_b operator} it follows that the operator $S_hNS_{-h}$ is generated by the function $S_h\tilde n$, $h\in\mathbb R^c$. Thus, the two functions $h\mapsto S_h\tilde n$ and $h\mapsto N\{h\}$ are connected by the isomorphism $U:\,\tilde n\mapsto N$, i.~e., $U:\,S_h\tilde n\mapsto N\{h\}$ for all $h\in\mathbb R^c$. Since $U$ is a topological isomorphism, the images of these two functions are relatively compact simultaneously. Thus, $\tilde n$ and $N$ are almost periodic simultaneously as well.
\end{proof}

\begin{corollary}\label{c:shift operator}
Let an operator $N\in\mathoo N_b$ be shift invariant, i.~e.,
\begin{equation*}
NS_h=S_hN,\qquad h\in\mathbb R^c.
\end{equation*}
Then $N$ is an operator of convolution with a function $g\in L_1\bigl(\mathbb R^c,\mathoo B(\mathbb E)\bigr)$.
\end{corollary}
\begin{proof}
By Propositions~\ref{p:N_b operator} and~\ref{p:ShNS-h:infty}, it is enough to prove that if a function $\tilde n\in\mathcal L_\infty\bigl(\mathbb R^c, L_1\bigl(\mathbb R^c,\mathoo B(\mathbb E)\bigr)\bigr)$ possesses the equalities (almost everywhere)
\begin{equation*}
S_h\tilde n=\tilde n,
\end{equation*}
for all $h\in\mathbb R^c$, then $\tilde n$ is a constant function almost everywhere.

We consider the function
\begin{equation*}
v(t,h)=\tilde n(t)-\tilde n(t-h),\qquad t,h\in\mathbb R^c.
\end{equation*}
It is measurable, see, e.g.,~\cite[Lemma 4.1.5]{Kurbatov99}. By assumption, for any $h\in\mathbb R^c$ the function $t\mapsto v(t,h)$ equals zero almost everywhere. Hence, by the Fubini theorem, the function $v$ equals zero almost everywhere. Therefore for almost all $t\in\mathbb R^c$ the function $h\mapsto v(t,h)$ equals zero almost everywhere. Let us take such a point $t$. It follows that $\tilde n(t-h)=\tilde n(t)$ for almost all $h$, which was to be proved.
\end{proof}

\begin{proposition}\label{p:N_{b}}
The set $\mathoo N_{b}\bigl(L_\infty\bigr)$ is a subalgebra of the algebra $\mathoo B(L_\infty)$.
\end{proposition}
\begin{proof}
Let $n,m\in\mathoo N_b$. Let $K\subset\mathbb R^c$ be a compact set. For almost all $x\in K$ we have
\begin{align}
\bigl(NMu\bigr)(x)&=\int_{\mathbb R^c}n(x,x-y)\,\bigl(Mu\bigr)(y)\,dy\notag\\
&=\int_{\mathbb R^c}n(x,x-y)\int_{\mathbb R^c}m(y,y-z)\,u(z)\,dz\,dy.\notag\\
\intertext{The function $(x,y,z)\mapsto n(x,x-y)\,m(y,y-z)\,u(z)$ is measurable and majorized by an absolutely integrable function on $K\times\mathbb R^c\times\mathbb R^c$; therefore, by the Fubini theorem, the order of integration may be changed. Thus, }
\bigl(NMu\bigr)(x)&=\int_{\mathbb R^c}\biggl(\int_{\mathbb R^c}n(x,x-y)\,m(y,y-z)\,dy\biggr)u(z)\,dz.\label{e:kernel k}
\end{align}
Again by the Fubini theorem, the function
\begin{equation*}
k(x,z)=\int_{\mathbb R^c}n(x,x-y)\,m(y,y-z)\,dy
\end{equation*}
defined by the internal integral in~\eqref{e:kernel k} exists for almost all $(x,z)\in K\times\mathbb R^c$; moreover, it is a summable function. Finally, for almost all $x$, we have
\begin{align*}
\int_{\mathbb R^c}\lVert k(x,z)\rVert\,dz&\le
\int_{\mathbb R^c}\biggl(\int_{\mathbb R^c}\lVert n(x,x-y)\rVert\,\lVert m(y,y-z)\rVert\,dy\biggr)\,dz\\
&=\int_{\mathbb R^c}\lVert n(x,x-y)\rVert\biggl(\int_{\mathbb R^c}\lVert m(y,y-z)\rVert\,dz\biggr)\,dy\\
&\le\lVert n\rVert_{\mathoo N_b}\,\lVert m\rVert_{\mathoo N_b},
\end{align*}
which means that $k\in\mathoo N_b$. From~\eqref{e:kernel k} it follows that the operator $NM$ is induced by the kernel $k$.
\end{proof}

\section{Integral operators with $L_1$-continuously varying kernels}\label{s:CN_1}
We denote by $\mathoo N_1=\mathoo N_1(\mathbb R^c,\mathbb E)$ the set of all measurable
functions $n:\,\mathbb R^c\times\mathbb R^c\to\mathoo B(\mathbb E)$ satisfying the
assumption: there exists a function $\beta\in\mathscr L_1(\mathbb R^c,\mathbb R)$ such that
\begin{equation}\label{e:est via beta}
\Vert n(x,y)\Vert\le\beta(y)
\end{equation}
for almost all $(x,y)\in\mathbb R^c\times\mathbb R^c$.
For convenience (without loss of generality), we assume that $\beta$ is defined
everywhere. Clearly, $\mathoo N_1\subset\mathoo N_b$. Kernels of the class $\mathoo N_1$ and operators induced by them were considered in~\cite{Beltita-Beltita15,Beltita-Beltita15e,Farrell-Strohmer10,Fendler-Leinert16,Kurbatov99,
Kurbatov-FDE01,Kurbatov-Kuznetsova16}. 

 \begin{proposition}[{\rm\cite[Proposition~5.4.3]{Kurbatov99}}]\label{p:5.4.3}
For any $n\in\mathoo N_1(\mathbb R^c,\mathbb E)$, the operator
 \begin{equation}\label{e:operator N}
\bigl(Nu\bigr)(x)=\int_{\mathbb R^c}n(x,x-y)\,u(y)\,dy
 \end{equation}
acts in $L_p(\mathbb R^c,\mathbb E)$ for all $1\le p\le\infty$. More precisely, for any
$u\in\mathscr L_p(\mathbb R^c,\mathbb E)$ the function $y\mapsto n(x,x-y)\,u(y)$ is
integrable for almost all $x$, and the function $Nu$ belongs to $\mathscr L_p(\mathbb
R^c,\mathbb E)${\rm;} if $u_1$ and $u_2$ coincide almost everywhere, then $Nu_1$ and
$Nu_2$ also coincide almost everywhere. Besides,
 \begin{equation}\label{e:norm of N}
\Vert N:\,L_p\to L_p\Vert\le\Vert\beta\Vert_{L_1}.
 \end{equation}
 \end{proposition}

We denote by $\mathoo N_1=\mathoo N_1(L_p)$, $1\le p\le\infty$, the set of all operators $N\in\mathoo B(L_p)$ of the form~\eqref{e:operator N} generated by $n\in\mathoo N_1(\mathbb R^c,\mathbb E)$.

 \begin{proposition}[{\rm\cite[Proposition~3.3]{Kurbatov-Kuznetsova16}}]\label{p:n and n_1}
If two functions $n,n_1\in\mathoo N_1(\mathbb R^c,\mathbb E)$ coincide almost everywhere
on $\mathbb R^c\times\mathbb R^c$, then they induce the same operator~\eqref{e:operator
N}.
 \end{proposition}

 \begin{theorem}[{\rm\cite[Theorem~5.4.7]{Kurbatov99}}]\label{t:5.4.7}
The subalgebra $\widetilde{\mathoo N_1}(L_p)$, $1\le p\le\infty$, is full in the algebra $\mathoo B(L_p)$ i.~e., if the operator $\mathbf1+N$, where $N\in\mathoo N_1(L_p)$, is
invertible, then $(\mathbf1+N)^{-1}=\mathbf1+M$, where $M\in\mathoo N_1(L_p)$.
 \end{theorem}

 \begin{corollary}[{\rm\cite[Corollary~5.4.8]{Kurbatov99}}]\label{c:5.4.8}
Let $n\in\mathoo N_1$, and the operator $N$ be defined by~\eqref{e:operator N}. If the
operator $\mathbf1+N$ is invertible in $L_p$ for some $1\le p\le\infty$, then it is
invertible in $L_p$ for all $1\le p\le\infty$. Moreover, the kernel $m$ of the operator
$M$, where $(\mathbf1+N)^{-1}=\mathbf1+M$, does not depend on $p$.
 \end{corollary}

We denote by $\mathoo C\mathoo N_1=\mathoo C\mathoo N_1(\mathbb R^c,\mathbb E)$ the class
of kernels $n\in\mathoo N_1$ such that the function $n$ can be redefined on a set of
measure zero so that it becomes defined everywhere, estimate~\eqref{e:est via beta} holds
for all $x$ and $y$, and the associated function $x\mapsto\bar n(x)$ becomes
continuous in the norm of $L_1\bigl(\mathbb R^c,\mathoo B(\mathbb E)\bigr)$. Unless otherwise stated, we will always assume that such an override has already been performed. We note that in this case, integral~\eqref{e:operator N} exists for \emph{all} $x\in\mathbb R^c$ provided $u\in L_\infty$.

\begin{proposition}\label{p:tilde n=bar n}
For any $n\in\mathoo N_1(\mathbb R^c,\mathbb E)$ the functions $\bar n:\mathbb R^c\to L_1\bigl(\mathbb R^c,\mathoo B(\mathbb E)\bigr)$ and $\tilde n:\mathbb R^c\to L_1\bigl(\mathbb R^c,\mathoo B(\mathbb E)\bigr)$ are continuous in the norm of $L_1\bigl(\mathbb R^c,\mathoo B(\mathbb E)\bigr)$ simultaneously. Thus, a kernel $n\in\mathoo N_1$ belongs to $\mathoo C\mathoo N_1=\mathoo C\mathoo N_1(\mathbb R^c,\mathbb E)$ if and only if the kernel $n\in\mathoo N_1$ can be redefined on a set of measure zero so that it becomes defined everywhere, estimate~\eqref{e:est via beta} holds
for all $x$ and $y$, and the associated function $x\mapsto\tilde n(x)$ is
continuous in the norm of $L_1\bigl(\mathbb R^c,\mathoo B(\mathbb E)\bigr)$.
\end{proposition}
\begin{proof}
The proof follows from the fact that for any function $g\in L_1\bigl(\mathbb R^c,\mathoo B(\mathbb E)\bigr)$ the family $S_hg$, where $\bigl(S_hg\bigr)(x)=g(x-h)$, depends on $h\in\mathbb R^c$ continuously in the norm of~$L_1$.
\end{proof}

We denote by $\mathoo C\mathoo N_1(L_p)$, $1\le p\le\infty$, the set of all operators $N\in\mathoo B(L_p)$ induced by kernels $n\in\mathoo C\mathoo N_1$ in accordance with formula~\eqref{e:operator N}.

 \begin{theorem}[{\rm\cite[Theorem~5.3]{Kurbatov-Kuznetsova16}}]\label{t:fin}
The subalgebra $\widetilde{\mathoo C\mathoo N_1}(L_p)$, $1\le p\le\infty$, is full in the algebra $\mathoo B(L_p)$, i.~e., if the operator $\mathbf1+N$, where $N\in\mathoo C\mathoo N_1(L_p)$, is invertible, then $(\mathbf1+N)^{-1}=\mathbf1+M$, where $M\in\mathoo C\mathoo N_1(L_p)$.
 \end{theorem}

\section{Integral operators with almost periodic kernels}\label{s:N_1,AP}

We denote by $\mathoo C\mathoo N_{1,\,AP}=\mathoo C\mathoo N_{1,\,AP}(\mathbb R^c,\mathbb E)$ the class of kernels $n\in\mathoo C\mathoo N_1$ such that the function $n$ can be redefined on a set of
measure zero so that it becomes defined everywhere, estimate~\eqref{e:est via beta} holds
for all $x$ and $y$, and the associated function $x\mapsto\tilde n(x)$ becomes
almost periodic in the norm of $L_1\bigl(\mathbb R^c,\mathoo B(\mathbb E)\bigr)$, i.~e., $\tilde n\in C_{AP}\bigl(\mathbb R^c,L_1\bigl(\mathbb R^c,\mathoo B(\mathbb E)\bigr)\bigr)$.

\begin{proposition}\label{p:5.1.9}
For any $n\in\mathoo C\mathoo N_{1,\,AP}(\mathbb R^c,\mathbb E)$, the continuous function $\tilde n:\,\mathbb R^c\to L_1\bigl(\mathbb R^c,\mathoo B(\mathbb E)\bigr)$ possesses a unique extension to a function $\tilde n\in C\bigl(\mathbb R_b^c,L_1\bigl(\mathbb R^c,\mathoo B(\mathbb E)\bigr)\bigr)$.
\end{proposition}
\begin{proof}
The proof follows from Proposition~\ref{t:5.1.9}.
\end{proof}

We denote by $\mathoo C\mathoo N_{1,\,AP}\bigl(L_p\bigr)$, $1\le p\le\infty$, the set of all operators $N\in\mathoo B(L_p)$ induced by kernels $n\in\mathoo C\mathoo N_{1,\,AP}$.

\begin{proposition}\label{p:N_{1,AP}}
The set $\mathoo C\mathoo N_{1,\,AP}\bigl(L_p\bigr)$ is a subalgebra of the algebra $\mathoo B(L_p)$, $1\le p\le\infty$.
\end{proposition}
\begin{proof}
The closedness with respect to addition and scalar multiplication is evident. We prove the closedness with respect to multiplication.
Let $M,N\in\mathoo C\mathoo N_{1,\,AP}\bigl(L_p\bigr)$. From Theorem~\ref{t:5.4.7} it follows that $M,N\in\mathoo C\mathoo N_{1}\bigl(L_p\bigr)$. It remains to prove that the operator $K=NM$ is induced by an almost periodic function $x\mapsto\tilde k(x)$.

Let $p=\infty$. By Corollary~\ref{c:N_b iso}, the operators $N,M\in\mathoo B(L_\infty)$ are almost periodic. Consequently, by Theorem~\ref{t:6.5.2}, the operator $K=NM\in\mathoo B(L_\infty)$ is almost periodic as well. Therefore, again by Corollary~\ref{c:N_b iso}, the associated function $\tilde k$ is almost periodic.
\end{proof}

 \begin{theorem}\label{t:main}
The subalgebra $\widetilde{\mathoo C\mathoo N}_{1,\,AP}(L_p)$, $1\le p\le\infty$, is full in the algebra $\mathoo B(L_p)$, i.~e., if the operator $\mathbf1+N$, where $N\in\mathoo C\mathoo N_{1,\,AP}(L_p)$, is invertible, then $(\mathbf1+N)^{-1}=\mathbf1+M$, where $M\in\mathoo C\mathoo N_{1,\,AP}(L_p)$.
 \end{theorem}
\begin{proof}
Since $\mathoo C\mathoo N_{1,\,AP}\subseteq\mathoo C\mathoo N_1$, by Theorem~\ref{t:fin} the operator $(\mathbf1+N)^{-1}$ can be represented in the form $\mathbf1+M$, where
\begin{equation*}
\bigl(Mu\bigr)(x)=\int_{\mathbb R^c}m(x,x-y)\,u(y)\,dy
\end{equation*}
and $m\in\mathoo C\mathoo N_1$. By the definition of the class $\mathoo C\mathoo N_1$, we can assume without loss of generality that $m$ is defined everywhere, estimate of the kind~\eqref{e:est via beta} holds
for all $x$ and $y$, and the associated function $x\mapsto\bar m(x)$ is
continuous in the norm of $L_1\bigl(\mathbb R^c,\mathoo B(\mathbb E)\bigr)$.

Since $\mathoo C\mathoo N_1\subseteq\mathoo N_1$, by Corollary~\ref{c:5.4.8} the operator $\mathbf1+N$ is invertible in $L_p$ for all $p$ and the inverse operator is defined by the formula $\mathbf1+M$, where the operator $M$ is generated by the same kernel $m\in\mathoo N_1$ for all $p$ as well. So, we consider the operators $\mathbf1+N$ and $\mathbf1+M$ as acting in $L_\infty$.

We recall~\cite[I.6.14]{Dunford-Schwartz-I:eng} that a subset $K$ of a Banach space is called \emph{totally bounded} if for every $\varepsilon>0$ it is possible to cover $K$ by a finite number of balls $B(k_i,\varepsilon)=\{\,x\in K:\,\lVert x-k_i\rVert<\varepsilon\,\}$, $i=1,\dots,n$, with centers $k_i\in K$. It is well known~\cite[I.6.15]{Dunford-Schwartz-I:eng} that a subset of a Banach space is relatively compact if and only if it is totally bounded.

By Theorem~\ref{t:6.5.2}, the operator $M:\,L_\infty\to L_\infty$ is almost periodic, i.~e., the family
\begin{equation*}
M\{h\}=S_hMS_{-h},\qquad h\in\mathbb R^c,
\end{equation*}
is relatively compact or, equivalently, totally bounded. Let, for a given $\varepsilon>0$, the set of balls
\begin{equation*}
B(M\{h_i\},\varepsilon)=\{\,M\{h\}:\,\lVert M\{h\}-M\{h_i\}\rVert<\varepsilon\,\},\qquad i=1,\dots,n,
\end{equation*}
forms a finite covering of the set $M\{h\}$, $h\in\mathbb R^c$.

By Proposition~\ref{p:ShNS-h:infty}, the operator $M\{h\}=S_hMS_{-h}$ is induced by the function $S_h\tilde m$. It follows from Proposition~\ref{p:N_b operator} that the balls
\begin{equation*}
B(S_{h_i}\tilde m,C\varepsilon)=\{\,S_{h}\tilde m:\,\lVert S_{h}\tilde m-S_{h_i}\tilde m\rVert_{L_1}<C\varepsilon\,\},\qquad i=1,\dots,n,
\end{equation*}
(where $C=1/c$) form a finite covering of the set $\{\,S_h\tilde m:\,h\in\mathbb R^c\,\}$.
Indeed,
\begin{align*}
\lVert S_{h}\tilde m-S_{h_i}\tilde m\rVert_{C(\mathbb R^c,L_1(\mathbb R^c,\mathoo B(\mathbb E)))}
&=\esssup_{x\in\mathbb R^c}\Vert\tilde m(x-h)-\tilde m(x-h_i)\Vert_{L_1}\\
&\le C\lVert S_hMS_{-h}-S_{h_i}MS_{-h_i}:\,L_\infty\to L_\infty\rVert\\
&=C\lVert M\{h\}-M\{h_i\}:\,L_\infty\to L_\infty\rVert.
\end{align*}
Since $C$ is a constant, the set $S_h\tilde m$, $h\in\mathbb R^c$, is totally bounded and, consequently, relatively compact. Thus, $m\in\mathoo C\mathoo N_{1,\,AP}$.
\end{proof}

\section{Almost periodic integral operators\\ with absolutely convergent Fourier series}\label{s:APWi}
We denote by $\mathoo C\mathoo N_{1,\,APW}=\mathoo C\mathoo N_{1,\,APW}(\mathbb R^c,\mathbb E)$ the class of kernels $n:\,\mathbb R^c\times\mathbb R^c\to\mathoo B(\mathbb E)$ that can be represented in the form
\begin{equation}\label{e:APW kernel}
n(x,y)=\sum_{k=1}^\infty e^{i\langle\omega_k,x\rangle}n_k(y),
\end{equation}
where $\omega_k\in\mathbb R^c$ and $n_k\in L_1\bigl(\mathbb R^c,\mathoo B(\mathbb E)\bigr)$, with
\begin{equation*}
\sum_{k=1}^\infty \lVert n_k\rVert_{L_1}<\infty.
\end{equation*}
Clearly,
\begin{equation*}
\mathoo C\mathoo N_{1,\,APW}\subset\mathoo C\mathoo N_{1,\,AP}\subset\mathoo C\mathoo N_1.
\end{equation*}
We note that for kernel~\eqref{e:APW kernel} we have
\begin{align*}
\tilde n(x)&=\sum_{k=1}^\infty e^{i\langle\omega_k,x\rangle}n_k,\\
\bigl(Nu\bigr)(x)&=\sum_{k=1}^\infty e^{i\langle\omega_k,x\rangle}\int_{\mathbb R^c}n_k(x-y)\,u(y)\,dy,\\
\bigl(N\{h\}u\bigr)(x)=\bigl(S_hNS_{-h}u\bigr)(x)&=\sum_{k=1}^\infty e^{i\langle\omega_k,x-h\rangle}\int_{\mathbb R^c}n_k(x-y)\,u(y)\,dy.
\end{align*}

We denote by $\mathoo C\mathoo N_{1,\,APW}\bigl(L_p\bigr)$, $1\le p\le\infty$, the set of all operators $N\in\mathoo B(L_p)$ induced by kernels $n\in\mathoo C\mathoo N_{1,\,APW}$. We say that an operator $N\in\mathoo C\mathoo N_{1,\,APW}\bigl(L_p\bigr)$ \emph{has an almost periodic kernel with an absolutely convergent Fourier series}.

For any $g\in L_1\bigl(\mathbb R^c,\mathoo B(\mathbb E)\bigr)$, we define the \emph{convolution operator}
\begin{equation}\label{e:G_g}
\bigl(G_gu\bigr)(x)=\int_{\mathbb R^c}g(x-y)u(y)\,dy.
\end{equation}
With this notation the operator $N$ induced by kernel~\eqref{e:APW kernel} can be represented in the form
\begin{equation}\label{e:N=sum Psi G}
N=\sum_{k=1}^\infty\Psi_{\omega_k}G_{n_k}.
\end{equation}
For brevity, instead of~\eqref{e:N=sum Psi G} we will use the representation (cf.~\eqref{e:N=sum Psi A'})
\begin{equation}\label{e:N=sum Psi G'}
N=\sum_{\omega\in\mathbb R^c}\Psi_{\omega}G_{\omega}
\end{equation}
for operators $N\in\mathoo C\mathoo N_{1,\,APW}\bigl(L_p\bigr)$, where $G_\omega$ are operators of convolution with functions $g_\omega\in L_1\bigl(\mathbb R^c,\mathoo B(\mathbb E)\bigr)$ and $\sum_{\omega\in\mathbb R^c}\lVert g_{\omega}\rVert_{L_1}<\infty$. We assume that only a countable number of operators $G_{\omega}$ in~\eqref{e:N=sum Psi G'} are nonzero. Clearly, $\mathoo C\mathoo N_{1,\,APW}\bigl(L_p\bigr)\subset\mathoo B_{APW}(L_p)$.
In particular, we have (see Proposition~\ref{p:N(H)=})
\begin{equation*}
N\{h\}=\sum_{\omega\in\mathbb R^c} e^{-i\langle\omega,h\rangle}\Psi_{\omega}G_{\omega},\qquad h\in\mathbb R^c.
\end{equation*}

\begin{proposition}\label{p:N_{1,APW}}
The set $\mathoo C\mathoo N_{1,\,APW}\bigl(L_p\bigr)$ is a subalgebra of the algebra $\mathoo B(L_p)$, $1\le p\le\infty$.
\end{proposition}
\begin{proof}
The closedness with respect to addition and scalar multiplication is evident. We prove the closedness with respect to multiplication. Let operators $N,M\in\mathoo B(L_p)$ be represented in the form
\begin{equation*}
N=\sum_{\omega\in\mathbb R^c}\Psi_{\omega}G_{\omega}\qquad\text{and}\qquad
M=\sum_{\nu\in\mathbb R^c}\Psi_{\nu}H_{\nu},
\end{equation*}
where $G_\omega$ and $H_\nu$ are operators of convolution with functions $g_\omega,h_\nu\in L_1\bigl(\mathbb R^c,\mathoo B(\mathbb E)\bigr)$ respectively, with $\sum_{\omega\in\mathbb R^c}\lVert g_{\omega}\rVert_{L_1}<\infty$ and $\sum_{\nu\in\mathbb R^c}\lVert h_{\nu}\rVert_{L_1}<\infty$. We have
\begin{equation*}
NM=\sum_{\omega\in\mathbb R^c}\sum_{\nu\in\mathbb R^c}\Psi_{\omega}G_{\omega}\Psi_{\nu}H_{\nu}
=\sum_{\omega\in\mathbb R^c}\sum_{\nu\in\mathbb R^c}\Psi_{\omega}\Psi_{\nu}\bigl(\Psi_{-\nu}G_{\omega}\Psi_{\nu}\bigr)H_{\nu}.
\end{equation*}
It remains to observe that $\Psi_{-\nu}G_{\omega}\Psi_{\nu}$ is an operator of convolution with the function $x\mapsto g_\omega(x)e^{-i\langle\nu,x\rangle}$ (evidently, the $L_1$-norm of this function is equal to the $L_1$-norm of the function $g_\omega$). Indeed,
\begin{align*}
\bigl(\Psi_{-\nu}G_{\omega}\Psi_{\nu}\bigr)(x)
&=e^{-i\langle\nu,x\rangle}\int_{\mathbb R^c}g_\omega(x-y)e^{i\langle\nu,y\rangle}u(y)\,dy\\
&=\int_{\mathbb R^c}\bigl(g_\omega(x-y)e^{-i\langle\nu,x-y\rangle}\bigr)u(y)\,dy.\qed
\end{align*}
\renewcommand\qed{}
\end{proof}

 \begin{theorem}\label{t:main2}
The subalgebra $\widetilde{\mathoo C\mathoo N}_{1,\,APW}(L_p)$, $1\le p\le\infty$, is full in the algebra $\mathoo B(L_p)$, i.~e., if the operator $\mathbf1+N$, where $N\in\mathoo C\mathoo N_{1,\,APW}(L_p)$, is invertible, then $(\mathbf1+N)^{-1}=\mathbf1+M$, where $M\in\mathoo C\mathoo N_{1,\,APW}(L_p)$.
 \end{theorem}

 \begin{proof}
By Theorems~\ref{t:main} and~\ref{t:fin}, the operator $M$ has the form
\begin{equation*}
\bigl(Mu\bigr)(x)=\int_{\mathbb R^c}m(x,x-y)\,u(y)\,dy,
\end{equation*}
with the same $m\in\mathoo C\mathoo N_{1,\,AP}$ for all $1\le p\le\infty$. So, it remains to prove that $m\in\mathoo C\mathoo N_{1,\,APW}$.

We suppose that $p=\infty$.

We extend the function $\tilde m\in C_{AP}\bigl(\mathbb R^c,L_1\bigl(\mathbb R^c,\mathoo B(\mathbb E)\bigr)\bigr)$ to the function $\tilde m\in C\bigl(\mathbb R^c_b,L_1\bigl(\mathbb R^c,\mathoo B(\mathbb E)\bigr)\bigr)$ in accordance with Corollary~\ref{p:5.1.9}. By Proposition~\ref{p:Kroneker}, this extension can be interpreted as the extension by continuity. Clearly, for any $\omega\in\mathbb R^c$ there exists the integral
\begin{equation*}
m_\omega=\int_{\mathbb X_b=\mathbb R^c_b}\langle\omega,\varkappa\rangle\tilde m(\varkappa)\,d\varkappa
\end{equation*}
with respect to the Haar measure.

We consider the family
\begin{equation*}
M\{h\}=S_hMS_{-h}:\,L_\infty\to L_\infty,\qquad h\in\mathbb R^c.
\end{equation*}
By Theorem~\ref{t:6.5.2}, the operator $M:\,L_\infty\to L_\infty$ is almost periodic, i.~e., $M\{\cdot\}\in C_{AP}\bigl(\mathbb R^c,\mathoo B(L_\infty)\bigr)$. Therefore, by Corollary~\ref{c:5.1.9}, it has a unique extension to $M\{\cdot\}\in C\bigl(\mathbb R^c_b,\mathoo B(L_\infty)\bigr)$. Again by Proposition~\ref{p:Kroneker}, this extension can be interpreted as the extension by continuity.

By Corollary~\ref{c:N_b iso}, the two functions $h\mapsto S_h\tilde m$ and $h\mapsto M\{h\}$ are connected by the isomorphism $U:\,\tilde m\mapsto M$, i.~e., $U:\,S_h\tilde m\mapsto M\{h\}$ for all $h$. Therefore their extensions on $\mathbb R^c_b$ by continuity are connected by $U$ as well. Consequently, the integrals
\begin{equation*}
\int_{\mathbb X_b}\langle\omega,\varkappa\rangle S_\varkappa\tilde m\,d\varkappa\qquad\text{and}\qquad \int_{\mathbb X_b}\langle\omega,\varkappa\rangle M\{\varkappa\}\,d\varkappa
\end{equation*}
are also connected by the same isomorphism $U:\,\tilde m\mapsto M$; here $\varkappa\mapsto S_\varkappa\tilde m$ means the extension by continuity of the function $h\mapsto S_h\tilde m$ from $\mathbb R^c$ to $\mathbb R^c_b$.

From Theorem~\ref{t:B_{APW} is full} it follows that the operator $(\mathbf1+N)^{-1}\{\varkappa\}$ has the form
\begin{equation*}
(\mathbf1+N)^{-1}\{\varkappa\}=\sum_{\omega\in\mathbb R^c}\langle-\omega,\varkappa\rangle \Psi_{-\omega}B_\omega,
\end{equation*}
where $B_\omega\in\mathoo A(L_\infty)$ and $\sum_{\omega\in\mathbb R^c}\lVert B_{\omega}\rVert<\infty$. Consequently,
\begin{equation*}
M\{\varkappa\}=\bigl[\mathbf1-(\mathbf1+N)^{-1}\bigr]\{\varkappa\}=\mathbf1-\sum_{\omega\in\mathbb R^c}\langle-\omega,\varkappa\rangle \Psi_{-\omega}B_\omega.
\end{equation*}

By Lemma~\ref{l:5.2.3}, for $\omega_0\in\mathbb R^c$, we have
\begin{equation}\label{e:int PsiB}
\int_{\mathbb X_b=\mathbb R^c_b}\langle\omega_0,\varkappa\rangle\Bigl(\mathbf1-\sum_{\omega\in\mathbb R^c}\langle-\omega,\varkappa\rangle \Psi_{-\omega}B_\omega\Bigr)\,d\varkappa
=
\begin{cases}
-\Psi_{-\omega_0}B_{\omega_0}&\text{for $\omega_0\neq0$},\\
\mathbf1-B_0 &\text{for $\omega_0=0$}.
\end{cases}
\end{equation}

On the other hand, by the above, the operator
\begin{equation*}
\int_{\mathbb X_b=\mathbb R^c_b}\langle\omega_0,\varkappa\rangle\Bigl(\mathbf1-\sum_{\omega\in\mathbb R^c}\langle-\omega,\varkappa\rangle \Psi_{-\omega}B_\omega\Bigr)\,d\varkappa
\end{equation*}
corresponds to the function (here the integrand $\varkappa\mapsto S_\varkappa\tilde m$ takes values in the space $C\bigl(\mathbb R^c_b,L_1\bigl(\mathbb R^c,\mathoo B(\mathbb E)\bigr)\bigr)$)
\begin{equation*}
\int_{\mathbb X_b=\mathbb R^c_b}\langle\omega_0,\varkappa\rangle S_\varkappa\tilde m\,d\varkappa,
\end{equation*}
where $\varkappa\mapsto S_\varkappa\tilde m$ means the extension by continuity of the function $h\mapsto S_h\tilde m$ from $\mathbb R^c$ to $\mathbb R^c_b$. Therefore operator~\eqref{e:int PsiB} is the integral operator $M_{\omega_0}$ generated by a function $\tilde m_{\omega_0}\in C\bigl(\mathbb R^c_b,L_1\bigl(\mathbb R^c,\mathoo B(\mathbb E)\bigr)\bigr)$.

Let $\omega_0\neq0$. Then $M_{\omega_0}=-\Psi_{-\omega_0}B_{\omega_0}$. Hence, the operator $-\Psi_{\omega_0}M_{\omega_0}=B_{\omega_0}$ belongs to $\mathoo A(L_\infty)$. Clearly, $\Psi_{\omega_0}M_{\omega_0}$ is an integral operator. Since it is shift invariant, by Corollary~\ref{c:shift operator} it is an operator of convolution with a function of the class $L_1$. The case $\omega_0=0$ is considered in a similar way.
 \end{proof}

\providecommand{\bysame}{\leavevmode\hbox to3em{\hrulefill}\thinspace}
\providecommand{\MR}{\relax\ifhmode\unskip\space\fi MR }
\providecommand{\MRhref}[2]{%
  \href{http://www.ams.org/mathscinet-getitem?mr=#1}{#2}
}
\providecommand{\href}[2]{#2}

\end{document}